\documentclass[11pt,a4paper]{article}

\usepackage{caption2}
\usepackage{CJK}
\usepackage{tabls}
\usepackage{bm}
\usepackage{multirow}
\usepackage{amssymb}
\usepackage{amsfonts}
\usepackage{mathrsfs}
\usepackage{empheq}
\usepackage{amsmath}
\usepackage{amsthm}
\usepackage{graphicx}
\usepackage{color}
\usepackage{indentfirst}
\usepackage{hyperref}
\usepackage{float}
\usepackage{cases}
\usepackage[titletoc]{appendix}
\usepackage[colorinlistoftodos,english]{todonotes}

\allowdisplaybreaks%

\def\e{\eqref}

\def\i1n{i=1,\cdots,n}
\def\j1n{j=1,\cdots,n}
\def\ij1n{i,j=1,\cdots,n}

\def\R{\mathbb R}

\def \i{\mathrm i}

 \numberwithin{equation}{section}
\theoremstyle{definition}

\def\R{{\bf R}}

\def\e{{\varepsilon}}

\newtheorem{thm}{Theorem}[section]

\newtheorem{lem}{Lemma}[section]

\newtheorem{rem}{Remark}[section]

\theoremstyle{definition}
\theoremstyle{theorem}
\theoremstyle{lemma}

\begin{document}

\begin{CJK*}{GB}{gbsn}
\title{Formation of Finite Time Singularity for Axially Symmetric Magnetohydrodynamic Waves in $3$-$D$}


\author{Lv Cai \thanks{School of Mathematical Sciences, Fudan University, Shanghai 200433, China (18110180023@fudan.edu.cn).}
\and Ning-An Lai \thanks{Corresponding author at: School of Mathematical Sciences, Zhejing Normal University, Jinhua 321004, China (ninganlai@zjnu.edu.cn).}}

\maketitle

\begin{abstract}
In this paper we study the compressible magnetohydrodynamics equations in three dimensions, which offer a good model for plasmas. Formation of singularity for $C^1$-solution in finite time is proved with axisymmetric initial data. The key observation is that the magnetic force term admits good structure with axisymmtric assumption.

\end{abstract}

\emph{keywords}: Singularity, Compressible, Magnetohydrodynamics, Axially Symmetric, Test Function


\section{Introduction}
\par

In this article we consider the compressible magnetohydrodynamics equations in $\R^3$, which can be used to describe the motion of a compressible perfectly conducting fluid in a magnetic field $B=(B^1, B^2, B^3)$. The equations can be stated as
\begin{equation}\label{mhd}
\begin{aligned}
\begin{split}
\left\lbrace
\begin{array}{lr}
\rho_{t} + \nabla \cdot (\rho u) =0, \\
\rho ( \partial_{t} u + u \cdot \nabla u ) + \nabla p - \mu^{-1} (\nabla \times B) \times B = 0 , \\
S_{t} + u \cdot \nabla S =0, \\
B_{t} - \nabla \times ( u \times B ) =0, \\
p=A \rho^{ \gamma } e^{S}, ~~~~~(A>0, \gamma>1),\\
\nabla \cdot B =0,\\
\end{array}	
\right.
\end{split}	
\end{aligned}
\end{equation}
where $ \rho,u=(u_1, u_2, u_3), p$ and $S$ are density, velocity, pressure and specific entropy of the fluid, respectively. The positive parameter $\mu$ denotes the magnetic permeability. The equation $\eqref{mhd}_6$ can be viewed as a constraint, if the magnetic field is divergence free initially then it will remain divergence free for all the time $t>0$. The state equation $\eqref{mhd}_5$ with adiabatic index $\gamma>1$ means that we are considering a polytropic fluid.

The smooth initial data
\begin{equation}\label{initialdata}
\begin{aligned}
\begin{split}
t=0:\left\lbrace
\begin{array}{lr}
\rho(0, x)=\overline{\rho}+\e\rho_0(x), \\
u(0, x)=\e u_0(x) , \\
B(0, x)=\e B_0(x), \\
S(0, x)=\overline{S}+\e S_0(x),\\
\end{array}	
\right.
\end{split}	
\end{aligned}
\end{equation}
satisfy
\begin{equation} \label{supp}
\begin{aligned}
supp~\left\{\rho_0(x), u_0(x), B_0(x), S_0(x)\right\}\subset \left\{x\big||x|\le L\right\}
\end{aligned}
\end{equation}
where $\overline{\rho}, L>0$ and $\overline{S}\in \R$ are constants, and $\e$ is a parameter to describe the smallness of the initial perturbations. Without loss of generality, we may assume that $\overline{\rho}=L=1$.

The compressible magnetohydrodynamic equations offer a good model of plasmas. A plasma is an ionized gas in which the flow of electrical charge plays a significant role in the dynamics. They couple the compressible Euler system to equations of motion for the magnetic field, which are derived from the Maxwell system and some simplifying approximation. We may refer the reader to the classic monograph (chapter $IV$, 37) \cite{Chan} for the detailed derivation.

The system \eqref{mhd} can be rewritten as a symmetric hyperbolic system (details will be presented below), which was introduced by Friedrichs \cite{F54}. As is well known, the general first order hyperbolic system admits a unique local solution in the Sobolev space $H^s$ and $s>\frac n2+1$ ($n$ denotes the space dimension), which can be extended as long as it remains bounded in $C^1$, see the pioneering work \cite{FL71, Lax71, FM72,Kat75a}. This kind of result leads to possible types of singularity from smooth initial data: (a)the solution itself may remain bounded, while the first derivatives become unbounded; (b)the solution itself grows without bound. We often say that a gradient blow-up or shock occurs in the former case, while that the solution blows up in the latter case.

There are many classic results about the shock formation for the equations of fluid mechanics in one space
dimension, and more generally for systems of conservation laws in one space dimension, see \cite{Lax64,Lax73,John74}. These results are established by characteristic method, under a certain structural assumption on the
system called genuine nonlinearity and suitable conditions on the initial data. The characteristic method seems not traceable for high dimensional system ($n\ge 2$). Even for the system \eqref{mhd} in $1$-$D$ (the unknowns are functions of time $t\ge 0$ and one space variable $x\in \R$), it is not easy to obtain shock formation by this method, since the MHD system has s linearly degenerate characteristic field, see \cite{Liu79}. In some sense this means that the shock develop for magnetohydrodynamics, that is, the mechanics of a perfectly electrically conducting fluid in the presence of a magnetic field remains open, see the conclusive introduction (page $21$) in the monograph \cite{Chris07}.

The first and general blow-up result (related to case $b$ mentioned above) in three dimensional fluids was obtained by Sideris \cite{Sid85}, in which he considered the compressible
Euler equations for ideal gas with adiabatic index $\gamma>1$ and
with initial data which coincide with constant states outside a ball. With some additional assumptions on the initial data a finite time blow-up of an averaged quantity of the solution (the density) was established. Similar blow-up results for $2$-$D$ compressible Euler system were obtained in \cite{Ram89} by using a similar method. In both \cite{Sid85} and \cite{Ram89}, the blow-up functional involves the departure of the density $\rho$ from
its value $\overline{\rho}$ in the constant state and a kernel which
is function of the distance from the center. Such kind of averaged quantity satisfies a one dimensional nonlinear wave equation, and then the one dimensional d'Alembertian plays a key role to obtain the desired blow-up result. The idea has been used to show finite time blow-up for $1$-$D$ MHD system \eqref{mhd} (the unknown functions depend on time $t\ge 0$ and one space variable $x\in \R$) in \cite{Ram94}. It works for MHD waves \eqref{mhd} in one space dimension due to the reason that the magnetic force term will be of the gradient of a positive quantity, which is not the case in the more difficult problem for MHD wave system in two and three space dimensions. Recently, the distance function kernel used in \cite{Sid85,Ram89} has been replaced by a special test function in \cite{JZ20, LXZ22}, and gave another method to show finite time blow-up for compressible Euler system for both $\gamma>1$ and $\gamma=1$. What is more, this method also works for the $2$-$D$ magnetohydrodynamics system \eqref{mhd} with a vertical magnetic field, thus,
\begin{equation}\label{2d}
\rho(t, x_1, x_2),~u(t, x_1, x_2)=(u_1, u_2, 0),~B(t, x_1, x_2)=(0, 0, b),
\end{equation}
see \cite{JZ20}, in which they found an interesting property of the $2$-$D$ solution \eqref{2d}, thus, $\frac{b}{\rho}$ takes a constant value along the characteristic $\frac{dx}{dt}=u$.

Coming to the physical case of three space dimensions for the MHD system \eqref{mhd}, to the best of our knowledge, there is still no formation of singularity result. Even for the blow-up result, the good structure of magnetic force in $1$-$D$ or property for $\frac b \rho$ in $2$-$D$ does not hold again in $3$-$D$. However, if we consider the axially symmetric solution
\begin{equation}\label{axisym}
\begin{aligned}
\begin{split}
\left\lbrace
\begin{array}{lr}
u (t,x) = u^{r}(t,r,z)e^{r} + u^{z}(t,r,z)e^{z} , \\
B (t,x) = B^{\theta}(t,r,z)e^{\theta}, \\
p (t,x) = p(t,r,z),\\
S(t, x)=S(t, r, z)\\
\end{array}	
\right.
\end{split}	
\end{aligned}
\end{equation}
to the system \eqref{mhd}, the magnetic force term in $\eqref{mhd}_2$ still admits a good structure: the gradient of a positive quantity plus a nonnegative term. This key observation then leads to our main blow-up result.

\begin{thm} \label{thm1}
	Assume that the initial data \eqref{initialdata} satisfy \eqref{supp} and
	\begin{equation}\label{initialS}
	S(0, x)\ge \overline{S},~~~x\in \R.
	\end{equation}
	If we further assume
	\begin{equation}
	\begin{aligned}
	\e\left(\int_{\mathbb{R}^{3}} \rho_{0} F(x) dx + \int_{\mathbb{R}^{3}} \nabla F\cdot u_0(x) dx \right)\triangleq C\e>0,
	\end{aligned}
	\end{equation}
	where $F$ is defined in \eqref{F} below,
	then the axisymmetric $C^1$-solution of \eqref{mhd} will blow up in a finite time, and the existence time $T$ is bounded from above by $\exp\left(C \epsilon^{-1}\right)$.
\end{thm}
Hereinafter, $C$ denotes a generic positive constant independent of $\e$, whose value may change from line to line.

\begin{rem}
	
	According to the work \cite{Sid91}, we believe that the above upper bound of lifespan is optimal, in the sense that we may obtain the same lower bound with respect to $\e$ if we consider the irrotational MHD system.
	
\end{rem}

\begin{rem}
	The upper bound of lifespan in Theorem \ref{thm1} holds for all $\gamma>1$, while for the $1$-$D$ case in \cite{Ram94}, the optimal upper bound of lifespan was established only for $\gamma=2$ and he stated that for the other $\gamma>1$, the lifespan will be different from that of $\gamma=2$. Also for the $2$-$D$ case in \cite{JZ20}, the upper bound of lifespan was established only for $\gamma=2$. Our method can give the same lifespan estimate from above for the system \eqref{mhd} with all $\gamma>1$ in both $\R$ and $\R^2$.
	
\end{rem}

\section{Local Existence and Finite Speed of Propagation}

The MHD system \eqref{mhd} can be written as a symmetric hyperbolic system in the form
\begin{equation}\label{sym}
\begin{aligned}
A_0(U)U_t+\sum_{i=1}^{3}A_i(U)U_{x_i}=0,
\end{aligned}
\end{equation}
where $U=(p, u, S, B)^{T}$, $A_0(U)$ is positive definite, and $A_i(U) (i=0, 1, 2, 3)$ are $8\times 8$ symmetric matrices. It is now well known that the symmetric hyperbolic system \eqref{sym} admits a unique local $C^1$-solution over $[0, T)$ for some $T>0$, provided the initial data are sufficiently regular, see \cite{Kat75a, Ma84}.

Denote the sound speed by
\[
\sigma(\rho, s)=\sqrt{\frac{\partial p}{\partial \rho}(\rho, S)}.
\]
With out loss of generality, we may set $\overline{\rho}=1$, $A=\left(\gamma e^{\overline{S}}\right)^{-1}$, so $\overline{\sigma}=\sigma(1, \overline{S})=1$, and we then have the following finite speed of propagation of the initial perturbation.
\begin{lem}\label{finite}
	Let $(\rho, u, S, B)\in C^1([0, T)\times \R^3)$ solve the MHD system \eqref{mhd}, then for $0\le t<T$
	\begin{equation}\label{finites}
	\begin{aligned}
	(\rho, u, S, B)=(1, 0, \overline{S}, 0),~~~|x|\ge t+1.
	\end{aligned}
	\end{equation}
\end{lem}
The proof of Lemma \ref{finite} can be done by local energy estimate over the truncated light cone, we refer to the reference \cite{Sid84} for details.

\section{Reformulation in Cylindrical Coordinate}

We will reformulate the system \eqref{mhd} in cylindrical coordinate. Denote a point in $\mathbb{R}^{3}$ by $x=(x_1, x_2, x_3)$, and the corresponding point in cylindrical coordinate by $(r, \theta, z)$ with
$ r = \sqrt{x_{1}^{2} + x_{2}^{2}} $ and
\begin{equation}
\begin{aligned}
\begin{split}
\left\lbrace
\begin{array}{lr}
e^{r} = ( \frac{x_{1}}{r}, \frac{x_{2}}{r} , 0 )^{T} = ( \cos \theta, \sin \theta, 0 )^{T}, \\
e^{\theta} = ( - \frac{x_{2}}{r}, \frac{x_{1}}{r} , 0 )^{T} ( -\sin \theta, \cos \theta, 0 )^{T}, \\
e^{z} = ( 0,0,1 )^{T}.
\end{array}	
\right.
\end{split}	
\end{aligned}
\end{equation}

Let
\begin{equation}
\begin{aligned}
\begin{split}
\left\lbrace
\begin{array}{lr}
u (t,x) = u^{r}(t,r,\theta, z)e^{r} + u^{\theta}(t,r,\theta, z)e^{\theta} + u^{z}(t,r, \theta,z)e^{z} , \\
B (t,x) = B^{r}(t,r,\theta,z)e^{r} + B^{\theta}(t,r,\theta,z)e^{\theta} + B^{z}(t,r,\theta,z)e^{z} , \\
p (t,x) = p(t,r,\theta,z),\\
S(t, x)=S(t, r, \theta, z).
\end{array}	
\right.
\end{split}	
\end{aligned}
\end{equation}
The relationship between the bases of these two different coordinate systems is
\begin{equation}
\begin{aligned}
\begin{split}
\left\lbrace
\begin{array}{lr}
e^{1} = \cos \theta e^{r} - \sin \theta e^{\theta}, \\
e^{2} = \sin \theta e^{r} + \cos \theta e^{\theta}, \\
e^{3} = e^{z}.
\end{array}	
\right.
\end{split}	
\end{aligned}
\end{equation}
Noting that
\begin{equation}
\begin{aligned}
\begin{split}
\left\lbrace
\begin{array}{lr}
\partial_{\theta} e^{r} = e^{\theta}, \\
\partial_{\theta} e^{\theta} = - e^{r},
\end{array}	
\right.
\end{split}	
\end{aligned}
\end{equation}
then direct computations yields
\begin{equation}\label{grau}
\begin{aligned}
\nabla u & = \partial_{r} u e^{r} + \frac{1}{r} \partial_{\theta} u e^{\theta} + \partial_{z} u e^{z} \\
& = \partial_{r} u^{r} e^{r} \otimes e^{r}+ \partial_{r} u^{\theta} e^{\theta} \otimes e^{r} + \partial_{r} u^{z} e^{z} \otimes e^{r} \\
& + \frac{1}{r} \left[ (  \partial_{\theta} u^{r} - u^{\theta})e^{r} \otimes e^{\theta} + (  \partial_{\theta} u^{\theta} + u^{r} )e^{\theta}\otimes e^{\theta} +\partial_{\theta} u^{z}e^{z}\otimes e^{\theta} \right] \\
& + \partial_{z} u^{r} e^{r} \otimes e^{z}+ \partial_{z} u^{\theta} e^{\theta} \otimes e^{z} + \partial_{z} u^{z} e^{z} \otimes e^{z},
\end{aligned}
\end{equation}
\begin{equation}
\begin{aligned}\label{ugrau}
u \cdot \nabla u & = u^{r}\partial_{r} u + u^{\theta} \cdot \frac{1}{r} \partial_{\theta} u + u^{z}\partial_{z} u \\
& =u^{r} \partial_{r} u^{r} e^{r} + u^{r} \partial_{r} u^{\theta} e^{\theta} + u^{r} \partial_{r} u^{z} e^{z} \\
& + \frac{1}{r} \left[ ( u^{\theta} \cdot \partial_{\theta} u^{r} - (u^{\theta})^{2})e^{r} + ( u^{\theta} \partial_{\theta} u^{\theta} + u^{\theta} u^{r} )e^{\theta} +u^{\theta} \partial_{\theta} u^{z}e^{z} \right] \\
& + u^{z} \partial_{z} u^{r} e^{r} + u^{z} \partial_{z} u^{\theta} e^{\theta} + u^{z} \partial_{z} u^{z} e^{z} \\
& = \left[ u^{r} \partial_{r} u^{r} + \frac{1}{r} ( u^{\theta} \cdot \partial_{\theta} u^{r} - (u^{\theta})^{2}) + u^{z} \partial_{z} u^{r}  \right] e^{r} \\
& + \left[ u^{r} \partial_{r} u^{\theta} + \frac{1}{r} ( u^{\theta} \cdot \partial_{\theta} u^{\theta} + u^{\theta}u^{r}) + u^{z} \partial_{z} u^{\theta} \right ] e^{\theta} \\
& +\left [ u^{r} \partial_{r} u^{z} + \frac{1}{r}  u^{\theta} \cdot \partial_{\theta} u^{z} + u^{z} \partial_{z} u^{z} \right ] e^{z},
\end{aligned}
\end{equation}
\begin{equation}\label{grap}
\begin{aligned}
\nabla p = \partial_{r} p e^{r} + \frac{1}{r} \partial_{\theta} p e^{\theta} + \partial_{z} p e^{z},
\end{aligned}
\end{equation}
\begin{equation}\label{divu}
\begin{aligned}
\nabla \cdot u & = \frac{1}{r} \partial_{r} (r u^{r}) + \frac{1}{r} \partial_{\theta} u^{\theta} + \partial_{z} u^{z} \\
& = \frac{1}{r} u^{r} + \partial_{r} u^{r}+ \frac{1}{r} \partial_{\theta} u^{\theta} + \partial_{z} u^{z}
\end{aligned}
\end{equation}
and
\begin{equation}
\begin{aligned}
\Delta u & =  \nabla \cdot (\nabla u ) \\
& = \left( e^{r} \partial_{r} + \frac{1}{r}  e^{\theta} \partial_{\theta} +  e^{z} \partial_{z} \right) \cdot \left(\partial_{r} u e^{r} + \frac{1}{r} \partial_{\theta} u e^{\theta} + \partial_{z} u e^{z} \right)\\
& = \partial_{r}^{2} u + \frac{1}{r}  \partial_{r} u + \frac{1}{r^{2}} \partial_{\theta}^{2} u + \partial_{z}^{2} u \\
& = \partial_{r}^{2} u^{r} e^{r} +\partial_{r}^{2} u^{\theta} e^{\theta} + \partial_{r}^{2} u^{z} e^{z} \\
& + \frac{1}{r} \partial_{r} u^{r} e^{r} + \frac{1}{r} \partial_{r} u^{\theta} e^{\theta} + \frac{1}{r} \partial_{r} u^{z} e^{z} \\
& + \frac{1}{r^{2}} \left( \partial_{\theta}^{2} u^{r} - 2 \partial_{\theta} u^{\theta} - u^{r} \right) e^{r} + \frac{1}{r^{2}}\left ( \partial_{\theta}^{2} u^{\theta} + 2 \partial_{\theta} u^{r} - u^{\theta} \right) e^{\theta} + \frac{1}{r^{2}} \partial_{\theta}^{2} u^{z}e^{z} \\
& +  \partial_{z}^{2} u^{r} e^{r} +\partial_{z}^{2} u^{\theta} e^{\theta} + \partial_{z}^{2} u^{z} e^{z} \\
& = \left[ \partial_{r}^{2} u^{r} + \frac{1}{r} \partial_{r} u^{r} + \frac{1}{r^{2}} ( \partial_{\theta}^{2} u^{r} - 2 \partial_{\theta} u^{\theta} - u^{r} ) + \partial_{z}^{2} u^{r} \right] e^{r}\\
& + \left[ \partial_{r}^{2} u^{\theta} + \frac{1}{r} \partial_{r} u^{\theta} + \frac{1}{r^{2}} ( \partial_{\theta}^{2} u^{\theta} + 2 \partial_{\theta} u^{r} - u^{\theta} ) + \partial_{z}^{2} u^{\theta} \right] e^{\theta}\\
& + \left[ \partial_{r}^{2} u^{z} + \frac{1}{r} \partial_{r} u^{z} + \frac{1}{r^{2}}  \partial_{\theta}^{2} u^{z}  + \partial_{z}^{2} u^{z} \right] e^{z}.
\end{aligned}
\end{equation}

For the cross product term $ \nabla \times B $, noting that

\begin{equation} \label{curlB}
\begin{aligned}
\nabla \times B = &( \partial_{x_{2}} B^{3} -  \partial_{x_{3}} B^{2}  ) e^{1} + ( \partial_{x_{3}} B^{1} -  \partial_{x_{1}} B^{3}  ) e^{2} \\
&+ ( \partial_{x_{1}} B^{2} -  \partial_{x_{2}} B^{1}  ) e^{3},
\end{aligned}
\end{equation}
\begin{equation}\label{B}
\begin{aligned}
B (t,x) & = B^{1}(t,x_{1},x_{2},x_{3})e^{1} +B^{2}(t,x_{1},x_{2},x_{3})e^{2} + B^{3}(t,x_{1},x_{2},x_{3})e^{3} \\
& = B^{r}(t,r,\theta,z)e^{r} + B^{\theta}(t,r,\theta,z)e^{\theta} + B^{z}(t,r,\theta,z)e^{z},
\end{aligned}
\end{equation}
and
\begin{equation}\label{BB}
\begin{aligned}
\begin{split}
\left\lbrace
\begin{array}{lr}
B^{1} =  B^{r} \cos \theta - B^{\theta} \sin \theta, \\
B^{2} =  B^{r} \sin \theta + B^{\theta} \cos \theta , \\
B^{3} = B^{z}.
\end{array}	
\right.
\end{split}	
\end{aligned}
\end{equation}
we obtain
\begin{equation}\label{curlB1}
\begin{aligned}
\nabla \times B =& \left( - \partial_{z}  B^{\theta} + \partial_{\theta}  B^{z} \cdot \frac{1}{r}\right) e^{r} + \left( \partial_{z}  B^{r} - \partial_{r}  B^{z} \right) e^{\theta} \\
&+ \left( \partial_{r} B^{\theta} - \partial_{\theta} B^{r}\cdot \frac{ 1}{r} +  \frac{B^{\theta}}{r} \right) e^{z}.
\end{aligned}
\end{equation}

In a similar way, we may further obtain
\begin{equation}\label{curlBB}
\begin{aligned}
&(\nabla \times B) \times B \\
=& \left[ \partial_{z}  B^{r} \cdot B^{z}  - \partial_{r}  B^{z}  \cdot B^{z} - \partial_{r}  B^{\theta}  \cdot B^{\theta} - \frac{(B^{\theta})^{2}}{r} + \partial_{\theta} B^{r}\cdot \frac{ B^{\theta}}{r} \right]	e^{r} \\
& + \left[  \partial_{z}  B^{\theta} \cdot B^{z} +  \partial_{r}  B^{\theta} \cdot B^{r} + \frac{B^{r}B^{\theta}}{r} - \partial_{\theta} B^{z}\cdot \frac{ B^{z}}{r} - \partial_{\theta} B^{r}\cdot \frac{ B^{r}}{r}\right ]e^{\theta} \\
& + \left[ - \partial_{z}  B^{\theta} \cdot B^{\theta} - \partial_{z}  B^{r} \cdot B^{r} + \partial_{r}  B^{z} \cdot B^{r} + \partial_{\theta} B^{z}\cdot \frac{ B^{\theta}}{r}\right ] e^{z}
\end{aligned}
\end{equation}
and
\begin{equation}\label{UB}
\begin{aligned}
u \times B = [ u^{\theta} B^{z}  - u^{z} B^{\theta} ]	e^{r} + [   u^{z} B^{r}  - u^{r} B^{z} ]e^{\theta} + [  u^{r} B^{\theta}  - u^{\theta} B^{r} ] e^{z},\\
\end{aligned}
\end{equation}
\begin{equation}
\begin{aligned}
&\nabla \times ( u \times B) \\
=& \left[ - \partial_{z} (u^{z} B^{r}  - u^{r} B^{z}) + \frac{1}{r} \partial_{\theta} ( u^{r} B^{\theta}  - u^{\theta} B^{r}) \right]e^{r} \\
& + \left[ \partial_{z}(u^{\theta} B^{z}  - u^{z} B^{\theta}) - \partial_{r} (u^{r} B^{\theta}  - u^{\theta} B^{r})\right ]e^{\theta} \\
& + \left[  \partial_{r} ( u^{z} B^{r}  - u^{r} B^{z}) + \frac{1}{r} ( u^{z} B^{r}  - u^{r} B^{z}) - \frac{1}{r} \partial_{\theta} ( u^{\theta} B^{z}  - u^{z} B^{\theta}) \right] e^{z}.
\end{aligned}
\end{equation}
In conclusion, the MHD system \eqref{mhd} can be rewritten in the cylindrical coordinate  as
\begin{equation} \label{mhdsim}
\begin{aligned}
\begin{split}
\left\lbrace
\begin{array}{lr}
\partial_{t} \rho + \frac{1}{r} \rho u^{r} + \partial_{r} (\rho u^{r}) + \partial_{z} (\rho u^{z}) + \frac{1}{r} \partial_{\theta} (\rho u^{\theta}) = 0, & \\
\rho \left( \partial_{t} u^{r} + u^{r} \partial_{r} u^{r} + u^{z} \partial_{z} u^{r} + \frac{1}{r} ( u^{\theta} \cdot \partial_{\theta} u^{r} - (u^{\theta})^{2}) \right) + \partial_{r} p & \\
+ \mu^{-1} \left[ \partial_{r}  B^{\theta}  \cdot B^{\theta} + \frac{(B^{\theta})^{2}}{r} - \partial_{z}  B^{r} \cdot B^{z}  + \partial_{r}  B^{z}  \cdot B^{z} - \partial_{\theta} B^{r}\cdot \frac{ B^{\theta}}{r} \right] = 0, & \\
\rho \left( \partial_{t} u^{\theta} + u^{r} \partial_{r} u^{\theta} + u^{\theta}u^{r}+ \frac{1}{r} ( u^{\theta} \cdot \partial_{\theta} u^{\theta}) + u^{z} \partial_{z} u^{\theta} \right) + \frac{1}{r} \partial_{\theta} p & \\
+ \mu^{-1} \left[ \partial_{z}  B^{\theta} \cdot B^{z} +  \partial_{r}  B^{\theta} \cdot B^{r} + \frac{B^{r}B^{\theta}}{r} - \partial_{\theta} B^{z}\cdot \frac{ B^{z}}{r} - \partial_{\theta} B^{r}\cdot \frac{ B^{r}}{r} \right] = 0 , & \\
\rho ( \partial_{t} u^{z} + u^{r} \partial_{r} u^{z} + u^{z} \partial_{z} u^{z} + \frac{1}{r}  u^{\theta} \cdot \partial_{\theta} u^{z} ) + \partial_{z} p & \\
+ \mu^{-1} [ \partial_{z}  B^{\theta} \cdot B^{\theta} + \partial_{z}  B^{r} \cdot B^{r} - \partial_{r}  B^{z} \cdot B^{r} - \partial_{\theta} B^{z}\cdot \frac{ B^{\theta}}{r}] = 0 ,& \\
\partial_{t} S + u^{r} \partial_{r} S + \frac{1}{r} u^{\theta} \partial_{\theta} S + u^{z} \partial_{z} S= 0 , & \\
\partial_{t} B^{r} - \left[ - \partial_{z} (u^{z} B^{r}  - u^{r} B^{z}) + \frac{1}{r} \partial_{\theta} ( u^{r} B^{\theta}  - u^{\theta} B^{r})\right ] = 0, & \\
\partial_{t} B^{\theta} -\left [ \partial_{z}(u^{\theta} B^{z}  - u^{z} B^{\theta}) - \partial_{r} (u^{r} B^{\theta}  - u^{\theta} B^{r}) \right] = 0, & \\
\partial_{t} B^{z} - \left[\partial_{r} ( u^{z} B^{r}  - u^{r} B^{z}) + \frac{1}{r} ( u^{z} B^{r}  - u^{r} B^{z}) - \frac{1}{r} \partial_{\theta} ( u^{\theta} B^{z}  - u^{z} B^{\theta})\right] = 0, & \\
\partial_{r} B^{r} + \frac{1}{r} B^{r} + \partial_{z} B^{z} + \frac{1}{r} \partial_{\theta} B^{\theta} = 0.&
\end{array}	
\right.
\end{split}	
\end{aligned}
\end{equation}

By the uniqueness of local existence result, if we assume initially $ u^{\theta}_{0} = B^{r}_{0} = B^{z}_{0} = 0 $, then $ u^{\theta} = B^{r} = B^{z} = 0 $ holds for all later time. If we further assume the solution is axisymmetric, i.e., in the form as \eqref{axisym}, then system \eqref{mhdsim} can be simplified as
\begin{equation} \label{2.1}
\begin{aligned}
\begin{split}
\left\lbrace
\begin{array}{lr}
\partial_{t} \rho + \frac{1}{r} \rho u^{r} + \partial_{r} (\rho u^{r}) + \partial_{z} (\rho u^{z}) = 0, & \\
\rho ( \partial_{t} u^{r} + u^{r} \partial_{r} u^{r} + u^{z} \partial_{z} u^{r} ) + \partial_{r} p + \mu^{-1} \left[ \partial_{r}  B^{\theta}  \cdot B^{\theta} + \frac{(B^{\theta})^{2}}{r} \right] = 0, & \\
\rho ( \partial_{t} u^{z} + u^{r} \partial_{r} u^{z} + u^{z} \partial_{z} u^{z} ) + \partial_{z} p + \mu^{-1} [ \partial_{z}  B^{\theta} \cdot B^{\theta} ] = 0, & \\
\partial_{t} S + u^{r} \partial_{r} S + u^{z} \partial_{z} S= 0 ,& \\
\partial_{t} B^{\theta} +  \partial_{z}( u^{z} B^{\theta}) + \partial_{r} (u^{r} B^{\theta} )  = 0,&
\end{array}	
\right.
\end{split}	
\end{aligned}
\end{equation}
which is equivalent to
\begin{equation}\label{mhdsim1}
\begin{aligned}
\begin{split}
\left\lbrace
\begin{array}{lr}
\partial_{t} \rho + \partial_{r} (\rho u^{r}) + \partial_{z} (\rho u^{z}) + \frac{1}{r} \rho u^{r}= 0, & \\
\partial_{t} (\rho u^{r}) + \partial_{r} ( \rho (u^r)^2 ) +\partial_{z} ( \rho u^{r}  u^{z} ) + \frac{1}{r} \rho  (u^{r})^{2} & \\
+ \partial_{r} p + \mu^{-1} \left[ \partial_{r}  B^{\theta}  \cdot B^{\theta} + \frac{(B^{\theta})^{2}}{r}\right ] =0, & \\
\partial_{t} (\rho u^{z}) + \partial_{z} ( \rho (u^{z})^2 ) +\partial_{r} ( \rho u^{z}  u^{r} ) + \frac{1}{r} \rho  u^{r}u^{z} & \\
+ \partial_{z} p + \mu^{-1} [ \partial_{z}  B^{\theta}  \cdot B^{\theta} ] =0, & \\
\partial_{t} S + u^{r} \partial_{r} S + u^{z} \partial_{z} S= 0 ,& \\
\partial_{t} B^{\theta} +  \partial_{z}( u^{z} B^{\theta}) + \partial_{r} (u^{r} B^{\theta} )  = 0. &
\end{array}	
\right.
\end{split}	
\end{aligned}
\end{equation}

\section{Proof of theorem $ \ref{thm1} $}

We are in position to prove the main theorem. Going back to the orthogonal coordinate system, \eqref{mhdsim1} is is equivalent to
\begin{equation} \label{eq}
\begin{aligned}
\begin{split}
\left\lbrace
\begin{array}{lr}
\partial_{t} \rho + \nabla \cdot (\rho u)= 0, & \\
\partial_{t} (\rho u) + \nabla \cdot (\rho u \otimes u ) + \nabla p + \mu^{-1}\left [ \nabla \frac{(B^{\theta} )^{2}}{2} + \frac{(B^\theta)^2}{r}  e^{r} \right]  = 0, & \\
\partial_tS+u\cdot \nabla S=0,& \\
\partial_t\left(\frac{B^\theta}r\right)+\nabla\cdot\left(\frac{B^\theta u}{r}\right)=0. &\\
\end{array}	
\right.
\end{split}	
\end{aligned}
\end{equation}
Before showing the proof, we first introduce a test function, which will be used to construct an integral averaged functional.
\begin{lem}\label{lem2}
	Let
	\begin{equation}\label{F}
	\begin{aligned}
	F(x) = \int_{S^{2}} e ^{ \omega \cdot x } d \omega. \ ( x \in  \R^3 ),
	\end{aligned}
	\end{equation}
	then $F(x)$ is a radial function satisfying
	\begin{equation}\label{Fpro}
	\begin{aligned}
	&  (i). \ F(x) > 0 ,~~F''(R)>0, \\
	&  (ii). \ F(x) \sim (1 + |x|)^{ - 1 } e^{|x|},\\
	&  (iii). \ F''_{RR} + \frac{2F_{R}}{R} - F = 0 , \\
	\end{aligned}
	\end{equation}
	where
	\begin{equation}
	\begin{aligned}
	|x| = \sqrt{r^{2} +z^{2} } \triangleq R .
	\end{aligned}
	\end{equation}
\end{lem}

\begin{proof}
	
	From the monograph \cite{LZ17} (see page 308), we know $F(x)$ is radial and can be written
	as
	\begin{equation}  \label{Fr}
	\begin{aligned}
	F(R)&=C\left(\int_{0}^{1}e^{Rw_{1}}dw_{1}+\int_{0}^{1}e^{-Rw_{1}}dw_{1}%
	\right)\\ &=\frac{C}{R}\left(e^{R}-e^{-R}\right), \end{aligned}
	\end{equation}
	with some positive constant $C$. Direct computations yield
	\begin{equation}\label{F2D}
	F''(R)=C\left[\frac{e^{R}-e^{-R}}{R}-\frac{2(e^{R}+e^{-R})}{R^{2}}%
	+\frac{2(e^{R}-e^{-R})}{R^{3}}\right],
	\end{equation}
	then $\eqref{Fpro}_1$ and $\eqref{Fpro}_2$ follow from \eqref{Fr} and \eqref{F2D}.
	Noting \eqref{F}, $F(x)$ satisfies
	\[
	\Delta F=F,
	\]
	which implies $\eqref{Fpro}_3$ due to the radial symmetry.
\end{proof}

With $F(x)$ in hand, we first show the proof of Theorem \ref{thm1} with $\gamma=2$, which appears in the state equation $\eqref{mhd}_5$. Set
\begin{equation}\label{XY}
\begin{aligned}
& X(t) = \int_{\mathbb{R}^{3}} F(x) (\rho (t,x) -1) dx = \int_{\mathbb{R}^{3}} \int_{S^{2}} e ^{ \omega \cdot x }  (\rho (t,x) -1) d \omega dx, \\
& Y(t)  =\int_{\mathbb{R}^{3}} \int_{S^{2}} e ^{ \omega \cdot x } \rho (t,x) (u \cdot \omega)d \omega dx.
\end{aligned}
\end{equation}
Multiplying $(\ref{eq})_{1}$ by $e ^{ \omega \cdot x }$ and integrating by parts we get
\begin{equation}\label{dX}
\begin{aligned}
& \int_{\R^{3}} \int_{S^{2}} (\partial_{t} (\rho-1) + \nabla \cdot (\rho u))  e ^{ \omega \cdot x } d \omega dx \\
=& \frac{d}{dt}  \int_{\R^{3}} \int_{S^{2}} (\rho-1) e ^{ \omega \cdot x } d \omega dx - \int_{\R^{3}} \int_{S^{2}} \rho e ^{ \omega \cdot x } (u \cdot \omega) d \omega dx \\
=& 0,
\end{aligned}
\end{equation}
which means
\begin{equation}\label{dXY}
\begin{aligned}
\frac{dX}{dt}=Y.
\end{aligned}
\end{equation}
Then multiplying $(\ref{eq})_{2}$ by $e ^{ \omega \cdot x } \omega $ and integrating by parts one has
\begin{equation}\label{dY}
\begin{aligned}
& \int_{\R^{3}} \int_{S^{2}}  e ^{ \omega \cdot x } \Bigg[\sum\limits_{i=1}^{3} \frac{\partial (\rho u_{i} \omega_{i}) }{\partial t} + \sum\limits_{i,k=1}^{3} \frac{\partial (\rho u_{i} u_{k} \omega_{i}) }{\partial x_{k}} + \sum\limits_{i=1}^{3} \omega_{i} \frac{\partial p(\rho) }{\partial x_{i}} \\
&  + \mu^{-1} \left(\omega \cdot \nabla \frac{(B^\theta)^{2}}{2} + \frac{(B^\theta)^2}{r} \omega \cdot e^{r}\right)
\Bigg]  d \omega dx \\
=& \frac{d}{dt} \int_{\R^{3}} \int_{S^{2}} \rho e ^{ \omega \cdot x } (u \cdot \omega) d \omega dx - \int_{\R^{3}} \int_{S^{2}} \rho e ^{ \omega \cdot x } (u \cdot \omega)^{2} d \omega dx    \\
& -\int_{\R^{3}} \int_{S^{2}}  e ^{ \omega \cdot x }(p-\overline{p})  d \omega dx \\
& - \mu^{-1} \int_{\R^{3}} \frac{(B^\theta)^2}{2} F(x) dx + \mu^{-1} \int_{\R^{3}} \frac{(B^\theta)^2}{r} \nabla F(x) \cdot e^{r} dx \\
= &0.
\end{aligned}
\end{equation}
As long as the velocity field $u$ is $C^1$, there exists the particle paths
\begin{equation} \label{particle}
\begin{aligned}
\begin{split}
\left\lbrace
\begin{array}{lr}
\frac{dx}{dt}=u(t, x), & \\
x(0, \alpha)=\alpha. &\\
\end{array}	
\right.
\end{split}	
\end{aligned}
\end{equation}
Then $\eqref{eq}_3$ implies that $S$ remains constant along the paths, and hence it holds by combining \eqref{initialS}
\begin{equation}\label{SS}
\begin{aligned}
S(t, x)\ge \overline{S},~~~(t, x)\in \R_+\times \R^3,
\end{aligned}
\end{equation}
which in turn yields
\begin{equation}\label{pp}
\begin{aligned}
p-\overline{p}-(\rho-1)\ge Ae^{\overline{S}}(\rho-1)^2
\end{aligned}
\end{equation}
with $A=\left(2e^{\overline{S}}\right)^{-1}$. On the other hand, since $F(x)$ is radial,
\begin{equation}\label{Fer}
\begin{aligned}
\nabla F\cdot e^r&=(\partial_rFe^r+\partial_zFe^z)\cdot e^r\\
&=\partial_rF\\
&=F'_R\frac{\partial R}{\partial r}\\
&=F'_R\times \frac rR.
\end{aligned}
\end{equation}
Inserting $\eqref{Fpro}_3$, \eqref{pp} and \eqref{Fer} into \eqref{dY}, we come to
\begin{equation}\label{dY1}
\begin{aligned}
\frac{dY}{dt} \ge& C\int_{\R^3}(\rho-1)^2F(x)dx+X(t)+\mu^{-1} \int_{\R^{3}} (B^\theta)^2\left(\frac{F}{2}-\frac{F'_R}{R}\right) dx \\
=& C\int_{\R^3}(\rho-1)^2F(x)dx+X(t)+(2\mu)^{-1} \int_{\R^{3}} (B^\theta)^2F''_{RR}dx\\
\ge&X(t)+C\int_{\R^3}(\rho-1)^2F(x)dx,
\end{aligned}
\end{equation}
where we use the fact that $F''_{RR}\ge 0$. For the last nonlinear term in \eqref{dY1}, by H\"{o}lder inequality we have
\begin{equation}\label{nonX}
\begin{aligned}
X^2(t)\le&\int_{\R^3}(\rho-1)^2F(x)dx\times\int_{\R^3}F(x)dx\\
\lesssim&\int_{\R^3}(\rho-1)^2F(x)dx\times \int_0^{t+1}(1+R)e^RdR\\
\lesssim&\int_{\R^3}(\rho-1)^2F(x)dx\times (1+t)e^t.
\end{aligned}
\end{equation}
Finally, we can get an ODE system by combining \eqref{dXY}, \eqref{dY1} and \eqref{nonX}
\begin{equation}
\begin{aligned}
\begin{split}
\left\lbrace
\begin{array}{lr}
X^{\prime}(t) = Y(t), & \\
Y^{\prime}(t) \geq \frac{CX^{2}(t)}{e^{t} (t+1)} + X(t) . &
\end{array}	
\right.
\end{split}	
\end{aligned}
\end{equation}
And hence the result in Theorem \ref{thm1} comes from the following lemma
\begin{lem}(Lemma 2.1 in \cite{JZ20})\label{lem5}
	Let $X(t)$ be a smooth function that satisfies the following inequalities:
	\[
	X''(t)\ge \frac{CX^2(t)}{e^t(t+R_0)^{\frac{n-1}2}}+X(t), ~~~t>0, 1\le n\le 3,
	\]
	with initial data $X(0) + X'(0) = \e x_0$. Here, $x_0$ is a positive constant, $\e>0$ is small enough, and then $X(t)$ will blow up in finite time. Furthermore,
	the lifespan of $X(t)$ is
	\[
	\begin{aligned}
	\begin{split}
	T_0=\left\lbrace
	\begin{array}{lr}
	C\e^{-1},~~n=1, & \\
	C\e^{-2},~~n=2, & \\
	e^{C\e^{-1}},~~n=3. &
	\end{array}	
	\right.
	\end{split}	
	\end{aligned}
	\]
	
\end{lem}

For the case $\gamma>2$, the inequality \eqref{pp} will be changed to
\begin{equation}\label{pp1}
\begin{aligned}
\begin{split}
p-\overline{p}-(\rho-1)\ge & Ae^{\overline{S}}\left[\rho^\gamma-1-\gamma(\rho-1)\right]\\
\sim&
\left\lbrace
\begin{array}{rl}
|\rho -1|^{\gamma}, &  \rho \rightarrow +\infty\\
|\rho -1|^{2}, & \rho \rightarrow 1\\
C(\gamma, \overline{S}),    &  \rho \rightarrow 0\\
\end{array}	
\right.\\
\ge& \widetilde{C}(\gamma, \overline{S})(\rho -1)^{2},
\end{split}	
\end{aligned}
\end{equation}
then the result in Theorem \ref{thm1} can be obtained by a similar way as that of $\gamma=2$.

It is more complicated for the case $1<\gamma<2$, due to the different behavior of the nonlinearity
\begin{equation}\label{N}
\begin{aligned}
p-\overline{p}-(\rho-1)\ge& Ae^{\overline{S}}\left[\rho^\gamma-1-\gamma(\rho-1)\right]\\
\triangleq &N(\rho-1)
\end{aligned}
\end{equation}
when $\rho \rightarrow +\infty$ and $\rho \rightarrow 0$. However, we can use the Orlicz spaces techniques introduced in \cite{LS23} to handle it. For the details we may refer the reader to \cite{LS23}, and we only make a sketch of the key steps for the convenience. It follows from \eqref{dY} and \eqref{N} that
\begin{equation}\label{N1}
\begin{aligned}
X''(t)=&Y'(t)\ge \int_{\R^3}(p-\overline{p})F(x)dx\\
\ge &\int_{\R^3}N(\rho-1)F(x)dx+X(t),
\end{aligned}
\end{equation}
which yields
\begin{equation}\label{N2}
\begin{aligned}
Z''(t)+2Z'(t)\ge \int_{\R^3}N(\rho-1)\psi(t, x)dx
\end{aligned}
\end{equation}
by setting
\[
Z(t)=e^{-t}X(t),~~~\psi(t, x)=e^{-t}F(x).
\]
Introducing a $N$ function (see \cite{KR61}) $\Upsilon: \R \rightarrow[0,+\infty)$
\begin{equation}\label{def:Y}
\Upsilon(x) := N(|x|) =(|x|+1)^\gamma-1 - \gamma|x|,
\end{equation}
which is continuous, even, convex and satisfies the conditions
\begin{equation*}
\lim_{x\rightarrow 0} \frac{\Upsilon(x)}{x} = 0,
\qquad
\lim_{|x|\rightarrow +\infty} \frac{\Upsilon(x)}{|x|} = +\infty.
\end{equation*}
The first key step is that for $\rho>0$ we have from \eqref{N2}
\begin{equation}\label{N3}
\begin{aligned}
Z''(t)+2Z'(t)\gtrsim \int_{\R^3}\Upsilon(\rho-1)\psi(t, x)dx.
\end{aligned}
\end{equation}
The second key step is to show
\begin{equation}\label{N4}
\begin{aligned}
\int_{\R^3}\Upsilon(\rho-1)\psi(t, x)dx\gtrsim (1+t)^{-1}\Upsilon(Z),
\end{aligned}
\end{equation}
by using the special properties of Orlicz norm. Finally we come to a key inequality by combining \eqref{N3} and \eqref{N4}
\begin{equation}\label{N5}
\begin{aligned}
Z''(t)+2Z'(t)\gtrsim (1+t)^{-1}\Upsilon(Z).
\end{aligned}
\end{equation}
And hence the result in Theorem \ref{thm1} comes from the following lemma
\begin{lem}(Lemma 4 in \cite{LS23})\label{lem:lizhou-var}
	Let $0\le \lambda \le 1$. Assume that $I \in C^2([0,+\infty);\R)$ satisfies
	\begin{equation}\label{eq:lz}
	I''(t) + I'(t) \gtrsim (1+t)^{-\lambda} N(I(t))
	\end{equation}
	where $N(p), N'(p)>0$ for $p>0$ and
	\begin{gather*}
	N(p) \approx
	\begin{cases}
	p^{1+\alpha} &\text{if $0\le p \le 1$,}
	\\
	p^{1+\beta} &\text{if $p > 1$,}
	\end{cases}
	\end{gather*}
	for some $\alpha,\beta>0$.
	Suppose also
	\begin{equation*}
	I(0) = \e >0, \qquad I'(0) \ge 0.
	\end{equation*}
	Then, $I(t)$ blows up in a finite time. Moreover, if $\e>0$ is small enough, the lifespan $T_\e$ of $I(t)$ satisfies the upper bound
	\begin{equation*}
	T_\e \le
	\left\{
	\begin{aligned}
	&C \e^{-\frac{\alpha}{1-\lambda}} &&\text{if $0\le\lambda<1$,}
	\\
	&\exp(C \e^{-\alpha}) &&\text{if $\lambda=1$,}
	\end{aligned}
	\right.
	\end{equation*}
	where $C$ is a positive constant dependent on $\alpha,\beta,\lambda$, but independent of $\e$.
\end{lem}

\section*{Acknowledgement}

The authors were partially supported by NSFC(12271487, 12171097).

\clearpage


\bibliographystyle{plain}
\end{CJK*}

\end{document}